\documentclass[12pt]{amsart}

\usepackage{amsmath}
\usepackage{amscd}
\usepackage{amssymb}
\usepackage{amsfonts}
\usepackage{amsthm}
\usepackage{enumerate}
\usepackage{hyperref}
\usepackage{color}
\usepackage{graphicx}

\theoremstyle{plain}
\newtheorem{thm}{Theorem}[section]

\newtheorem{prop}[thm]{Proposition}
\newtheorem{lem}[thm]{Lemma}
\newtheorem{cor}[thm]{Corollary}

\theoremstyle{definition}

\newtheorem{dfns-rems}[thm]{Definitions and Remarks}
\newtheorem{notas-rems}[thm]{Notations and Remarks}
\newtheorem{exmps-rems}[thm]{Examples and Remarks}

\DeclareMathOperator{\depth}{depth}

\begin{document}


\title[Depth of symbolic powers of edge ideals]{On the depth of symbolic powers of edge ideals of graphs}


\author[S. A. Seyed Fakhari]{S. A. Seyed Fakhari}

\address{S. A. Seyed Fakhari, School of Mathematics, Statistics and Computer Science,
College of Science, University of Tehran, Tehran, Iran.}

\email{aminfakhari@ut.ac.ir}

\begin{abstract}
Assume that $G$ is a graph with edge ideal $I(G)$ and star packing number $\alpha_2(G)$. We denote the $s$-th symbolic power of $I(G)$ by $I(G)^{(s)}$. It is shown that the inequality $\depth S/(I(G)^{(s)})\geq \alpha_2(G)-s+1$ is true for every chordal graph $G$ and every integer $s\geq 1$. Moreover, it is proved that for any graph $G$, we have $\depth S/(I(G)^{(2)})\geq \alpha_2(G)-1$.
\end{abstract}


\subjclass[2000]{Primary: 13C15, 13F55, 05E40}


\keywords{Chordal graphs, Depth, Edge ideal, Star packing number, Symbolic powers}


\thanks{}


\maketitle


\section{Introduction} \label{sec1}

Let $\mathbb{K}$ be a field and $S = \mathbb{K}[x_1,\ldots,x_n]$  be the
polynomial ring in $n$ variables over $\mathbb{K}$. Computing and finding bounds for the depth (or equivalently, projective dimension) of homogenous ideals of $S$ and their powers have been studied by  several authors (see e.g., \cite{b'}, \cite{chhktt}, \cite{ds}, \cite{fm}, \cite{htt}, \cite{hh''}, \cite{hktt}, \cite{nt}).

In \cite{fhm}, Fouli, H${\rm \grave{a}}$ and Morey introduced the notion of {\it initially regular sequence}. Using this notion, they provided a method for estimating the depth of a homogenous ideal. To be more precise, let $I\subset S$ be a homogenous ideal and let $\{b_{i,j} \mid 1\leq i\leq q, 0\leq j\leq t_i\}$ be a subset of distinct variables of $S$. Suppose ${\rm in_<(I)}$ is the initial ideal of $I$ with respect to a fixed monomial order $<$ and assume that $G({\rm in}_<(I))=\{u_1, \ldots, u_m\}$ is the set of minimal monomial generators of ${\rm in}_<(I)$. It is shown in \cite[Theorem 3.11]{fhm} that $\depth S/I\geq q$, provided that the following conditions hold.
\begin{itemize}
\item[(i)] The monomials $u_1, u_2, \ldots, u_m$ are not divisible by $b_{i,j}^2$ for $1\leq i\leq q$ and $1\leq j\leq t_i$.
\item[(ii)] For $i=1, 2, \ldots, q$, if a monomial in $\{u_1, \ldots, u_m\}$ is divisible by $b_{i,0}$, then it is also divisible by $b_{i,j}$, for some integer $1\leq j\leq t_i$,
\end{itemize}

In Section \ref{sec2}, we provide an alternative proof for this result (see Proposition \ref{depin}). Our proof is based on a short exact sequence argument, while in \cite{fhm}, the authors construct an initially regular sequence to prove their result.

Fouli, H${\rm \grave{a}}$ and Morey \cite{fhm1} observed that the above result provides a combinatorial lower bound for the depth of edge ideals of graphs. Indeed, for every graph $G$ with edge ideal $I(G)$, we have$$\depth S/I(G)\geq \alpha_2(G),$$where $\alpha_2(G)$ denotes the so-called star packing number of $G$ (see Section \ref{sec2} for the definition of star packing number and see Corollary \ref{spn} for more details about the above inequality). It is proven in \cite[Theorem 3.7]{fhm1} that the above inequality can be extended to powers of $I(G)$ when $G$ is a forest. More precisely, for every forest $G$ and for every integer $s\geq 1$, the inequality
\[
\begin{array}{rl}
\depth S/I(G)^s \geq \alpha_2(G)-s+1
\end{array} \tag{$\dag$} \label{dag}
\]
holds. On the other hand, we know from \cite[Theorem 5.9]{svv} that for every forest $G$, the $s$-th ordinary and symbolic powers of $I(G)$ coincide. Hence, inequality (\ref{dag}) essentially says that for every forest $G$ and any positive integer $s$,
\[
\begin{array}{rl}
\depth S/I(G)^{(s)} \geq \alpha_2(G)-s+1.
\end{array} \tag{$\ddag$} \label{ddag}
\]

In Theorem \ref{main}, we generalize \cite[Theorem 3.7]{fhm1} by proving inequality (\ref{ddag}) for any chordal graph. Moreover, we show that inequality (\ref{ddag}) is true for $s=2$ and for any graph $G$ (see Theorem \ref{2power}).


\section{Preliminaries and known results} \label{sec2}

In this section, we provide the definitions and the known results which will be used in the next sections.

Let $G$ be a simple graph with vertex set $V(G)=\big\{x_1, \ldots,
x_n\big\}$ and edge set $E(G)$. For a vertex $x_i$, the {\it neighbor set} of $x_i$ is $N_G(x_i)=\{x_j\mid x_ix_j\in E(G)\}$. We set $N_G[x_i]=N_G(x_i)\cup \{x_i\}$ and call it the {\it closed neighborhood} of $x_i$. The cardinality of $N_G(x_i)$ is the {\it degree} of $x_i$ and will be denoted by ${\rm deg}_G(x_i)$. For every subset $U\subset V(G)$, the graph $G\setminus U$ has vertex set $V(G\setminus U)=V(G)\setminus U$ and edge set $E(G\setminus U)=\{e\in E(G)\mid e\cap U=\emptyset\}$. A subgraph $H$ of $G$ is called {\it induced} provided that two vertices of $H$ are adjacent if and only if they are adjacent in $G$. A graph $G$ is called {\it chordal} if it has no induced cycle of length at least four. A subset $W$ of $V(G)$ is a {\it clique} of $G$ if every two distinct vertices of $W$ are adjacent in $G$. A vertex $x$ of $G$ is a {\it simplicial vertex} if $N_G(x)$ is a clique. It is well-known that every chordal graph has a simplicial vertex. A subset $C$ of $V(G)$ is a {\it vertex cover} of $G$ if every edge of $G$ is incident to at least one vertex of $C$. A vertex cover $C$ is a {\it minimal vertex cover} if no proper subset of $C$ is a vertex cover of $G$. The set of minimal vertex covers of $G$ will be denoted by $\mathcal{C}(G)$. A subset $A$ of $V(G)$ is called an {\it independent subset} of $G$ if there are no edges among the vertices of $A$. Obviously, $A$ is independent if and only if $V(G)\setminus A$ is a vertex cover of $G$.

The {\it edge ideal} of a graph $G$ is defined as$$I(G)=\big(x_ix_j \, |\,  x_ix_j\in E(G)\big).$$For a subset $C$ of $\big\{x_1, \ldots, x_n\big\}$, we denote by $\mathfrak{p}_C$, the monomial prime ideal which is generated by the variables belonging to $C$. It is well-known that for every graph $G$,$$I(G)=\bigcap_{C\in \mathcal{C}(G)}\mathfrak{p}_C.$$

Let $I$ be an ideal of $S$ and let ${\rm Min}(I)$ denote the set of minimal primes of $I$. For every integer $s\geq 1$, the $s$-th {\it symbolic power} of $I$,
denoted by $I^{(s)}$, is defined to be$$I^{(s)}=\bigcap_{\frak{p}\in {\rm Min}(I)} {\rm Ker}(S\rightarrow (S/I^s)_{\frak{p}}).$$Let $I$ be a squarefree monomial ideal in $S$ and suppose that $I$ has the irredundant
primary decomposition $$I=\frak{p}_1\cap\ldots\cap\frak{p}_r,$$ where every
$\frak{p}_i$ is an ideal generated by a subset of the variables of
$S$. It follows from \cite[Proposition 1.4.4]{hh} that for every integer $s\geq 1$, $$I^{(s)}=\frak{p}_1^s\cap\ldots\cap
\frak{p}_r^s.$$We set $I^{(s)}=S$, for any integer $s\leq 0$.

It is clear that for any graph $G$ and every integer $s\geq 1$,$$I(G)^{(s)}=\bigcap_{C\in \mathcal{C}(G)}\mathfrak{p}_C^s.$$

As it was mentioned in introduction, Fouli, H${\rm \grave{a}}$ and Morey \cite{fhm} detected a method to bound the depth of a homogenous ideal. We provide an alternative proof for their result. Recall that for every monomial $u$ and for every variable $x_i$, the degree of $u$ with respect to $x_i$ is denoted by ${\rm deg}_{x_i}(u)$.

\begin{prop} [\cite{fhm}, Theorem 3.11] \label{depin}
Let $I$ be a proper homogenous ideal of $S$ and let $<$ be a monomial order. Assume that $B=\{b_{i,j} \mid 1\leq i\leq q, 0\leq j\leq t_i\}$ is a subset of distinct variables of $S$, such that the following conditions are satisfied.
\begin{itemize}
\item[(i)] For every pair of integers $1\leq i\leq q$, $1\leq j\leq t_i$ and for every $u\in G({\rm in}_<(I))$, we have ${\rm deg}_{b_{i,j}}(u)\leq 1$.
\item[(ii)] For $i=1, 2, \ldots, q$, if a monomial $u\in G({\rm in}_<(I))$ is divisible by $b_{i,0}$, then it is also divisible by $b_{i,j}$, for some integer $1\leq j\leq t_i$.
\end{itemize}
Then $\depth S/I\geq q$.
\end{prop}

\begin{proof}
It is known that $\depth S/I\geq \depth S/{\rm in}_<(I)$ (see e.g., \cite[Theorem 3.3.4]{hh}). Hence, replacing $I$ by ${\rm in}_<(I)$, we may suppose that $I$ is a monomial ideal. We use induction on $|B|$. There is nothing to prove for $|B|=0$, as in this case $q=0$. Therefore, assume that $|B|\geq 1$. If $t_i=0$, for every $i=1, 2, \ldots, q$, then it follows from condition (ii) that $b_{1,0}, \ldots, b_{q,0}$ do not divide the minimal monomial generators of $I$. In particular, they form a regular sequences on $S/I$ and the assertion follows. Thus, suppose that $t_i\geq 1$, for some $i$ with $1\leq i\leq q$. Without lose of generality, suppose $i=1$. Consider the following short exact sequence.
\begin{align*}
0\longrightarrow S/(I:b_{1,t_1})\longrightarrow S/I\longrightarrow S/(I,b_{1,t_1})\longrightarrow 0
\end{align*}
This yields that
\[
\begin{array}{rl}
\depth S/I \geq \min \big\{\depth S/(I:b_{1,t_1}), \depth S/(I,b_{1,t_1})\big\}.
\end{array} \tag{1} \label{ast}
\]
By condition (i), the variable $b_{1,t_1}$ does not appear in the minimal monomial generators of $(I:b_{1,t_1})$. In particular, $b_{1,t_1}$ is a regular element on $S/(I:b_{1,t_1})$. Let $S'$ be the polynomial ring obtained from $S$ by deleting the variable $b_{1,t_1}$ (in other words, $S'\cong S/(b_{1,t_1})$). Set $I':=(I:b_{1,t_1})\cap S'$. It follows that$$\depth S/(I:b_{1,t_1})=\depth S/((I:b_{1,t_1}), b_{1,t_1})+1=\depth S'/I'+1.$$Clearly, $I'$ satisfies the assumptions with respect to the set $\{b_{i,j} \mid 2\leq i\leq q, 0\leq j\leq t_i\}$ of variables. Thus, the induction hypothesis implies that $\depth S'/I'\geq q-1$. Hence, we deduce from the above equalities that$$\depth S/(I:b_{1,t_1})\geq q.$$

Using inequality (\ref{ast}), it suffices to prove that $\depth S/(I,b_{1,t_1})\geq q$. Set $I'':=I\cap S'$. Then $S/(I,b_{1,t_1})\cong S'/I''$. Put $t_1':=t_1-1$ and $t_i':=t_i$, for $i=2, \ldots, q$. Obviously, $I''$ satisfies the assumptions with respect to the set $\{b_{i,j} \mid 1\leq i\leq q, 0\leq j\leq t_i'\}$ of variables. Therefore, we conclude from the induction hypothesis that$$\depth S/(I,b_{1,t_1})=\depth S'/I''\geq q.$$
\end{proof}

Let $G$ be a graph and $x$ be a vertex of $G$. The subgraph ${\rm St}(x)$ of $G$ with vertex set $N_G[x]$ and edge set $\{xy\, |\, y\in N_G(x)\}$ is called a {\it star with center $x$}. A {\it star packing} of $G$ is a family $S$ of stars in $G$ which are pairwise disjoint, i.e., $V({\rm St}(x))\cap V({\rm St}(x'))=\emptyset$, for ${\rm St}(x), {\rm St}(x')\in S$. The quantity$$\max\big\{|S|\, |\, S \ {\rm is \ a \ star \ packing \ of} \ G\big\}$$ is called the {\it star packing number} of $G$. Following \cite{fhm1}, we denote the star packing number of $G$ by $\alpha_2(G)$.

The following corollary is an immediate consequence of Proposition \ref{depin}, and it was indeed observed in \cite{fhm1}.

\begin{cor} [\cite{fhm1}] \label{spn}
For every graph $G$, we have$$\depth S/I(G)\geq \alpha_2(G).$$
\end{cor}

\begin{proof}
Let $b_{1,0}, \ldots, b_{q,0}$ be the centers of stars in a largest star packing of $G$. Moreover, for $1\leq i\leq q$, assume that $N_G(b_{i,0})=\{b_{i,1}, \ldots, b_{i,t_i}\}$. Then the assumptions of Proposition \ref{depin} are satisfied and it follows that$$\depth S/I(G)\geq q=\alpha_2(G).$$
\end{proof}


\section{Symbolic powers of edge ideals of chordal graphs} \label{sec3}

In this section, we prove the first main result of this paper, Theorem \ref{main} which states that inequality (\ref{ddag}) is true for every chordal graph $G$ and for any integer $s\geq 1$. In order to prove this result, we first need to estimate the star packing number of the graph obtained from $G$ by deleting a certain subset of its vertices. This will be done in the following two lemmas.

\begin{lem} \label{packdel1}
Let $G$ be a graph and let $W$ be a subset of $V(G)$. Then for every $A\subseteq \bigcup_{x\in W}N_G[x]$, we have$$\alpha_2(G\setminus A)\geq \alpha_2(G)-|W|.$$
\end{lem}

\begin{proof}
Let $\mathcal{S}$ be the set of the centers of stars in a largest star packing of $G$. In particular, $|\mathcal{S}|=\alpha_2(G)$. Since every vertex in $A$ belongs to the closed neighborhood of a vertex in $W$, it follows from the definition of star packing that $|\mathcal{S}\cap A|\leq | W|$. Then the stars in $G\setminus A$ centered at the vertices in $\mathcal{S}\setminus A$ form a star packing in $G\setminus A$ of size at least $\alpha_2(G)-| W|$. Therefore, $\alpha_2(G\setminus A)\geq \alpha_2(G)-|W|$.
\end{proof}

\begin{lem} \label{packdel}
Assume that $G$ is a graph and $W=\{x_1, \ldots, x_d\}$ is a clique of $G$. Let $A$ be a subset of $V(G)$ such that

\begin{itemize}
\item [(i)] $A\subseteq \bigcup_{i=1}^dN_G(x_i)$,
\item [(ii)] $N_G(x_1)\setminus \{x_2, \ldots, x_d\} \subseteq A$, and
\item [(iii)] $x_1\notin A$.
\end{itemize}
Then $\alpha_2(G\setminus A)\geq \alpha_2(G)-d+1$.
\end{lem}

\begin{proof}
Let $\mathcal{S}$ be the set of the centers of stars in a largest star packing of $G$. Similar to the proof of the Lemma \ref{packdel1}, we have $|\mathcal{S}\cap A|\leq d$. If $|\mathcal{S}\cap A|\leq d-1$, then the stars in $G\setminus A$ centered at the vertices in $\mathcal{S}\setminus A$ form a star packing in $G\setminus A$ of size at least $\alpha_2(G)-d+1$. Thus, the assertion follows in this case. Therefore, suppose $|\mathcal{S}\cap A|=d$. In this case, we have$$x_1, \ldots, x_d\in \bigcup_{x\in \mathcal{S}\cap A}N_G(x).$$It again follows from the definition of star packing that$$x_1, \ldots, x_d\notin \bigcup_{x\in \mathcal{S}\setminus A}N_{G\setminus A}(x).$$Therefore, we conclude from condition (ii) that
\begin{align*}
N_{G\setminus A}[x_1]\cap \big(\bigcup_{x\in \mathcal{S}\setminus A}N_{G\setminus A}(x)\big)\subseteq \{x_1, \ldots, x_d\}\cap \big(\bigcup_{x\in \mathcal{S}\setminus A}N_{G\setminus A}(x)\big)=\emptyset.
\end{align*}
As a consequence, the stars in $G\setminus A$ centered at the vertices in $(\mathcal{S}\setminus A)\cup\{x_1\}$ form a star packing in $G\setminus A$ of size $\alpha_2(G)-d+1$. This completes the proof of the lemma.
\end{proof}

We are now ready to prove that inequality (\ref{ddag}) holds for any chordal graph. Indeed, we are able to prove the following stronger result.

\begin{prop} \label{sum1}
Let $G$ be a chordal graph. Suppose $H$ and $H'$ are subgraphs of $G$ with$$E(H)\cap E(H')=\emptyset \ \ \ and \ \ \ E(H)\cup E(H')=E(G).$$Assume further that $H$ is a chordal graph. Then for every integer $s\geq 1$,$$\depth S/(I(H)^{(s)}+I(H'))\geq \alpha_2(G)-s+1.$$
\end{prop}

\begin{proof}
As the isolated vertices have no effect on edge ideals, we assume that $V(H)=V(H')=V(G)$ (i.e., we extend the vertex sets of $H$ and $H'$ to $V(G)$). We use induction on $s+|E(H)|$. For $s=1$, we have $I(H)^{(s)}+I(H')=I(G)$ and the assertion follows from Corollary \ref{spn}. Therefore, suppose $s\geq 2$. If $E(H)=\emptyset$, then $I(H')=I(G)$ and again we have the required inequality by Corollary \ref{spn}. Hence, we assume $|E(H)|\geq 1$.

To simplify the notations, we set $I:=I(H)^{(s)}+I(H')$. Since $H$ is a chordal graph, it has a simplicial vertex, say $x_1$, with nonzero degree. Without loss of generality, suppose $N_H(x_1)=\big\{x_2, \ldots, x_d\big\}$, for some integer $d\geq 2$. Consider the following short exact sequence.
\begin{align*}
0\longrightarrow \frac{S}{(I:x_1\ldots x_d)}\longrightarrow \frac{S}{I}\longrightarrow \frac{S}{(I,x_1\ldots x_d)}\longrightarrow 0
\end{align*}
Using depth Lemma \cite[Proposition 1.2.9]{bh}, we have
\[
\begin{array}{rl}
\depth S/I \geq \min \big\{\depth S/(I:x_1\ldots x_d), \depth S/(I,x_1\ldots x_d)\big\}.
\end{array} \tag{2} \label{1}
\]

By assumption, for every pair of integers $i\neq j$, with $1\leq i,j\leq d$ we have $x_ix_j\in E(H)$. Therefore, $x_ix_j$ is not an edge of $H'$. Set$$U:=\bigcup_{i=1}^dN_{H'}[x_i]$$and$$U':=\bigcup_{i=1}^dN_{H'}(x_i).$$Then using \cite[Lemma 2]{s9}, we have
\begin{align*}
& (I:x_1\ldots x_d)=\big((I(H)^{(s)}+I(H')):x_1\ldots x_d\big)\\ & =I(H)^{(s-d+1)}+I(H'\setminus U) +\big({\rm the \ ideal\ generated\ by}\ U'\big)\\ & =I(H\setminus U')^{(s-d+1)}+I(H'\setminus U)+\big({\rm the \ ideal\ generated\ by}\ U'\big).
\end{align*}
This yields that$$\depth S/(I:x_1\ldots x_d)=\depth S'/(I(H\setminus U')^{(s-d+1)}+I(H'\setminus U)),$$where $S'=\mathbb{K}[x_i: 1\leq i\leq n, i\notin U']$. Let $G'$ be the union of $H\setminus U'$ and $H'\setminus U$. In fact, $G'$ is the induced subgraph of $G$ on $V(G)\setminus U'$. Clearly, $N_G(x_1)\setminus \{x_2, \ldots, x_d\}$ is contained in $U'$. Then the above equality together with Lemma \ref{packdel} and the induction hypothesis implies that
\[
\begin{array}{rl}
\depth S/(I:x_1\ldots x_d)\geq \alpha_2(G\setminus U')-(s-d+1)+1\geq \alpha_2(G)-s+1.
\end{array} \tag{3} \label{2}
\]

Using inequalities (\ref{1}) and (\ref{2}), it is enough to prove that$$\depth S/(I,x_1\ldots x_d)\geq \alpha_2(G)-s+1.$$For every integer $k$ with $1\leq k\leq d-1$, let $J_k$ be the ideal generated by all the squarefree monomials of degree $k$ on variables $x_2, \ldots, x_d$. We continue in the following steps.

\vspace{0.3cm}
{\bf Step 1.} Let $1\leq k\leq d-2$ be a fixed integer and assume that $\{u_1, \ldots, u_t\}$ is the set of minimal monomial generators of $x_1J_k$. In particular, every $u_j$ is divisible by $x_1$ and ${\rm deg}(u_j)=k+1$. For every integer $j$ with $1\leq j\leq t$, we prove that
\begin{align*}
& \depth S/(I+x_1J_{k+1}+(u_1, \ldots, u_{j-1}))\\ & \geq\min\big\{\depth S/(I+x_1J_{k+1}+(u_1, \ldots, u_j)), \alpha_2(G)-s+1\big\}.
\end{align*}
(Note that for $j=1$, we have $I+x_1J_{k+1}+(u_1, \ldots, u_{j-1})=I+x_1J_{k+1}$.)

Consider the following short exact sequence.
\begin{align*}
0 & \longrightarrow \frac{S}{(I+x_1J_{k+1}+(u_1, \ldots, u_{j-1})):u_j}\longrightarrow \frac{S}{I+x_1J_{k+1}+(u_1, \ldots, u_{j-1})}\\ & \longrightarrow \frac{S}{I+x_1J_{k+1}+(u_1, \ldots, u_j)}\longrightarrow 0
\end{align*}
As a consequence,
\begin{align*}
& \depth S/(I+x_1J_{k+1}+(u_1, \ldots, u_{j-1}))\geq\\& \min\big\{\depth S/((I+x_1J_{k+1}+(u_1, \ldots, u_{j-1})):u_j), \depth S/(I+x_1J_{k+1}+(u_1, \ldots, u_j))\big\}.
\end{align*}
Therefore, to complete this step, we need to show that$$\depth S/((I+x_1J_{k+1}+(u_1, \ldots, u_{j-1})):u_j)\geq \alpha_2(G)-s+1.$$

Set$$U_j:=\{x_i\mid 1\leq i \leq d \ {\rm and} \ x_i\ {\rm does\ not\ divide\ }u_j\}.$$For any $x_i\in U_j$, the monomial $x_iu_j$ is a squarefree monomial of degree $k+2$. Hence, $x_iu_j$ belongs to $x_1J_{k+1}$. This shows that$$({\rm the\ ideal\ generated\ by}\ U_j)\subseteq \big((x_1J_{k+1}+(u_1, \ldots, u_{j-1})):u_j\big).$$We show the reverse inclusion holds too.

Since $x_1J_{k+1}+(u_1, \ldots, u_{j-1})$ is a squarefree monomial ideal, it follows that$$\big((x_1J_{k+1}+(u_1, \ldots, u_{j-1})):u_j\big)$$is also a squarefree monomial ideal. On the other hand, $\{x_1, \ldots, x_d\}$ is the set of variables appearing in the set of minimal monomial generators of $x_1J_{k+1}+(u_1, \ldots, u_{j-1})$. This implies that every monomial generator of $\big((x_1J_{k+1}+(u_1, \ldots, u_{j-1})):u_j\big)$ is a squarefree monomial over the variables $x_1, \ldots, x_d$. Assume that $v$ is a minimal generator of $\big((x_1J_{k+1}+(u_1, \ldots, u_{j-1})):u_j\big)$. If $v$ is not equal to any of the variables belonging to $U_j$, then by definition of $U_j$, every variable dividing $v$, also divides $u_j$. As $v$ is a squarefree monomial, we have $v\mid u_j$. Since$$u_jv\in x_1J_{k+1}+(u_1, \ldots, u_{j-1}),$$we deduce that$$u_j^2\in x_1J_{k+1}+(u_1, \ldots, u_{j-1}),$$ which implies that$$u_j\in x_1J_{k+1}+(u_1, \ldots, u_{j-1}),$$because $x_1J_{k+1}+(u_1, \ldots, u_{j-1})$ is a squarefree monomial ideal. This is contradiction, as the degree of $u_j$ is strictly less that the degree of any monomial in $x_1J_{k+1}$ and moreover none of the monomials $u_1, \ldots, u_{j-1}$ is equal to $u_j$. Hence,
\[
\begin{array}{rl}
\big((x_1J_{k+1}+(u_1, \ldots, u_{j-1})):u_j\big)=({\rm the\ ideal\ generated\ by}\ U_j).
\end{array} \tag{4} \label{3}
\]

Let $W_j$ be the set of variables dividing $u_j$. In other words, $W_j=\{x_1, \ldots, x_d\}\setminus U_j$. We remind that for any pair of integers $1\leq i, j\leq d$, the vertices $x_i$ and $x_j$ are not adjacent in $H'$. Set$$U_j':=\bigcup_{x_i\in W_j}N_{H'}[x_i]$$and$$U_j'':=\bigcup_{x_i\in W_j}N_{H'}(x_i).$$Using equality (\ref{3}), we conclude that
\begin{align*}
& \big((I+x_1J_{k+1}+(u_1, \ldots, u_{j-1})):u_j\big)=\\ & \big((I(H)^{(s)}+I(H')+x_1J_{k+1}+(u_1, \ldots, u_{j-1})):u_j\big)=\\ & (I(H)^{(s)}:u_j)+I(H'\setminus U_j')+({\rm the\ ideal\ generated\ by}\ U_j\cup U_j'')\\ & =\big(I\big(H\setminus (U_j\cup U_j'')\big)^{(s)}:u_j\big)+I\big(H'\setminus (U_j\cup U_j')\big)\\ & +({\rm the\ ideal\ generated\ by}\ U_j\cup U_j'').
\end{align*}

Set $H_j:=H\setminus (U_j\cup U_j'')$ and $H_j':=H'\setminus (U_j\cup U_j')$. Then $H_j$ is a chordal graph, and $x_1$ is a simplicial vertex of $H_j$. It is also clear that $N_{H_j}[x_1]$ is the set of variables divining $u_j$. It thus follows from \cite[Lemma 2]{s9} and the above equalities that
\begin{align*}
& \big((I+x_1J_{k+1}+(u_1, \ldots, u_{j-1})):u_j\big)=I(H_j)^{(s-k)}+I(H_j')\\ & +({\rm the\ ideal\ generated\ by}\ U_j\cup U_j'').
\end{align*}
This yields that$$\depth S/((I+x_1J_{k+1}+(u_1, \ldots, u_{j-1})):u_j)=\depth S_j/(I(H_j)^{(s-k)}+I(H_j')),$$where $S_j=\mathbb{K}[x_i: 1\leq i\leq n, i\notin U_j\cup U_j'']$. Let $G_j$ be the union of $H_j$ and $H_j'$. Then $G_j$ is the induced subgraph of $G$ on $V(G)\setminus (U_j\cup U_j'')$. We conclude from Lemma \ref{packdel} (by considering the clique $W_j$) that$$\alpha_2(G_j)\geq \alpha_2(G)-| W_j| +1=\alpha_2(G)-k,$$where the last equality follows from the fact that ${\rm deg}(u_j)=k+1$. Hence, the induction hypothesis implies that
\begin{align*}
& \depth S/((I+x_1J_{k+1}+(u_1, \ldots, u_{j-1})):u_j)=\depth S_j/(I(H_j)^{(s-k)}+I(H_j'))\\ & \geq \alpha_2(G_j)-(s-k)+1\geq \alpha_2(G)-s+1,
\end{align*}
and this step is complete.

\vspace{0.3cm}
{\bf Step 2.} Let $1\leq k\leq d-2$ be a fixed integer. By a repeated use of Step 1, we have
\begin{align*}
& \depth S/(I+x_1J_{k+1}) \geq\min\big\{\depth S/(I+x_1J_{k+1}+(u_1, \ldots, u_t)), \alpha_2(G)-s+1\big\}\\ & =\min\big\{\depth S/(I+x_1J_k), \alpha_2(G)-s+1\big\}.
\end{align*}

\vspace{0.3cm}
{\bf Step 3.} It follows from Step 2 that
\begin{align*}
& \depth S/(I,x_1\ldots x_d)=\depth S/(I+x_1J_{d-1})\\ & \geq\min\big\{\depth S/(I+x_1J_{d-2}), \alpha_2(G)-s+1\big\}\\ & \geq\min\big\{\depth S/(I+x_1J_{d-3}), \alpha_2(G)-s+1\big\}\\ & \geq \cdots \geq\min\big\{\depth S/(I+x_1J_1), \alpha_2(G)-s+1\big\}.
\end{align*}
In particular,
\[
\begin{array}{rl}
\depth S/(I,x_1\ldots x_d)\geq\min\big\{\depth S/\big(I+(x_1x_2, x_1x_3, \ldots, x_1x_d)\big), \alpha_2(G)-s+1\big\}.
\end{array} \tag{5} \label{4}
\]

\vspace{0.3cm}
{\bf Step 4.} Let $L$ be the graph obtained from $H$, by deleting the edges $x_1x_2, \ldots, x_1x_d$. Then $L$ is the disjoint union of $H\setminus x_1$ and the isolated vertex $x_1$. In particular, $L$ is a chordal graph. Also, let $L'$ be the graph obtained from $H'$, by adding the edges $x_1x_2, \ldots, x_1x_d$. Then$$E(L)\cap E(L')=\emptyset \ \ \ and \ \ \ E(L)\cup E(L')=E(G).$$It follows from \cite[Lemma 3.2]{s8} and the induction hypothesis that
\begin{align*}
& \depth S/(I+(x_1x_2, x_1x_3, \ldots, x_1x_d))=\\ & \depth S/(I(H)^{(s)}+I(H')+(x_1x_2, x_1x_3, \ldots, x_1x_d)\big)=\\ & \depth S/((I(L)^{(s)}+I(L')))\geq \alpha_2(G)-s+1.
\end{align*}
Finally, inequality (\ref{4}) implies that
\[
\begin{array}{rl}
\depth S/(I,x_1\ldots x_d)\geq \alpha_2(G)-s+1.
\end{array} \tag{6} \label{5}
\]
Now, inequalities (\ref{1}), (\ref{2}) and (\ref{5}) complete the proof of the proposition.
\end{proof}

The following theorem is the main result of this section and follows easily from Proposition \ref{sum1}.

\begin{thm} \label{main}
Let $G$ be a chordal graph. Then for every integer $s\geq 1$, we have$$\depth S/(I(G)^{(s)})\geq\alpha_2(G)-s+1.$$
\end{thm}

\begin{proof}
The assertion follows from Proposition \ref{sum1} by substituting $H=G$ and $H'=\emptyset$.
\end{proof}


\section{Second symbolic power of edge ideals} \label{sec4}

The aim of this section is to show that inequality (\ref{ddag}) is true for $s=2$, Theorem \ref{2power}. To prove this result, we need to bound the depth of ideals of the form $\big(I(G)^{(k)}:xy\big)$, where $xy$ is an edge of $G$. To achieve this goal, we will use the following lemma in the case of $k=2$.

\begin{lem} \label{intsec}
Let $G$ be a graph and $xy$ be an edge of $G$. Then for any integer $k\geq 2$, we have$$\big(I(G)^{(k)}:xy\big)=\big(I(G)^{(k-1)}:x\big)\cap \big(I(G)^{(k-1)}:y\big).$$
\end{lem}

\begin{proof}
Let $u$ be a monomial in $\big(I(G)^{(k)}:xy\big)$. Then $uxy\in I(G)^{(k)}$. Clearly, this implies that $ux\in I(G)^{(k-1)}$. Therefore, $u\in \big(I(G)^{(k-1)}:x\big)$. Similarly, $u$ belongs to $\big(I(G)^{(k-1)}:y\big)$. Hence,$$\big(I(G)^{(k)}:xy\big)\subseteq\big(I(G)^{(k-1)}:x\big)\cap \big(I(G)^{(k-1)}:y\big).$$

To prove the reverse inclusion, let $v$ be a monomial in$$\big(I(G)^{(k-1)}:x\big)\cap \big(I(G)^{(k-1)}:y\big).$$We must show that $vxy\in I(G)^{(k)}$. It is enough to prove that for any minimal vertex cover $C$ of $G$, we have $vxy\in \mathfrak{p}_C^k$. So, let $C$ be a minimal vertex cover of $G$. It follows from $xy\in E(G)$ that $C$ contains at least one of the vertices $x$ and $y$. Without lose of generality, suppose $x\in C$. Since $v\in \big(I(G)^{(k-1)}:y\big)$, we have $vy\in I(G)^{(k-1)}\subseteq \mathfrak{p}_C^{k-1}$. This together with $x\in \mathfrak{p}_C$ implies that $vxy\in \mathfrak{p}_C^k$.
\end{proof}

The following theorem is the second main result of this paper.

\begin{thm} \label{2power}
For any graph $G$, we have$$\depth S/I(G)^{(2)}\geq \alpha_2(G)-1.$$
\end{thm}

\begin{proof}
Set $I:=I(G)$ and let $G(I)=\{u_1, \ldots, u_m\}$ be the set of minimal monomial generators of $I$. Using \cite[Theorem 4.12]{b}, we may assume that for every pair of integers $1\leq j< i\leq m$, one of the following conditions holds.
\begin{itemize}
\item [(i)] $(u_j:u_i) \subseteq (I^2:u_i)\subseteq (I^{(2)}:u_i)$; or
\item [(ii)] there exists an integer $k\leq i-1$ such that $(u_k:u_i)$ is generated by a subset of variables, and $(u_j:u_i)\subseteq (u_k:u_i)$.
\end{itemize}
For every integer $i$ with $1\leq i\leq m$ consider the short exact sequence
\begin{align*}
0 & \longrightarrow \frac{S}{(I^{(2)}+(u_1, \ldots, u_{i-1})):u_i}\longrightarrow \frac{S}{I^{(2)}+(u_1, \ldots, u_{i-1})}\\ & \longrightarrow \frac{S}{I^{(2)}+(u_1, \ldots, u_i)}\longrightarrow 0.
\end{align*}
It follows from depth Lemma \cite[Proposition 1.2.9]{bh} that
\begin{align*}
& \depth S/(I^{(2)}+(u_1, \ldots, u_{i-1}))\\ & \geq \min\big\{\depth S/((I^{(2)}+(u_1, \ldots, u_{i-1})):u_i), \depth S/(I^{(2)}+(u_1, \ldots, u_i))\big\}.
\end{align*}
Consequently,
\begin{align*}
& \depth S/I^{(2)}\geq\\ & \min\big\{\depth S/(I^{(2)}+I), \depth S/((I^{(2)}+(u_1, \ldots, u_{i-1})):u_i)\mid 1\leq i\leq m\big\}=\\ & \min\big\{\depth S/I, \depth S/((I^{(2)}+(u_1, \ldots, u_{i-1})):u_i)\mid 1\leq i\leq m\big\} \geq\\ & \min\big\{\alpha_2(G), \depth S/((I^{(2)}+(u_1, \ldots, u_{i-1})):u_i)\mid 1\leq i\leq m\big\},
\end{align*}
where the last inequality follows from Corollary \ref{spn}. Hence, it is enough to show that$$\depth S/((I^{(2)}+(u_1, \ldots, u_{i-1})):u_i)\geq \alpha_2(G)-1,$$for every integer $i$ with $1\leq i\leq m$.

Fix an integer $i$ with $1\leq i\leq m$ and assume that $u_i=xy$. We know from (i) and (ii) above that
\[
\begin{array}{rl}
\big((I^{(2)}, u_1, \ldots, u_{i-1}):u_i\big)=(I^{(2)}:u_i)+({\rm some \ variables}).
\end{array} \tag{7} \label{6}
\]
Let $A$ be the set of variables appearing in $\big((I^{(2)}, u_1, \ldots, u_{i-1}):u_i\big)$. Assume that $x\in A$. This means that $x^2y$ belongs to the ideal $(I^{(2)}, u_1, \ldots, u_{i-1})$. Since $u_1, \ldots, u_{i-1}$ do not divide $x^2y$, we deduce that $x^2y\in I(G)^{(2)}$. But this is a contradiction, as $C:=V(G)\setminus \{x\}$ is a vertex cover of $G$ with $x^2y\notin \mathfrak{p}_C^2$. Therefore, $x\notin A$. Similarly, $y\notin A$. It follows from $x,y\notin A$ and equality (\ref{6}) that
\begin{align*}
& \big((I^{(2)}, u_1, \ldots, u_{i-1}):u_i\big)=\big(I^{(2)}:u_i\big)+({\rm the\ ideal\ generated\ by}\ A)\\ & =\big(I^{(2)}+ ({\rm the\ ideal\ generated\ by}\ A)):u_i\big)\\ & =\big(I(G\setminus A)^{(2)}+({\rm the\ ideal\ generated\ by}\ A)):u_i\big)\\ & =\big(I(G\setminus A)^{(2)}:u_i\big)+({\rm the\ ideal\ generated\ by}\ A).
\end{align*}
Therefore,
\[
\begin{array}{rl}
\depth S/((I^{(2)},u_1, \ldots, u_{i-1}):u_i)=\depth S_A/(I(G\setminus A)^{(2)}:u_i),
\end{array} \tag{8} \label{7}
\]
where $S_A=\mathbb{K}[x_i: 1\leq i\leq n, x_i\notin A]$. It follows from Lemma \ref{intsec} that$$(I(G\setminus A)^{(2)}:u_i)=(I(G\setminus A):x)\cap(I(G\setminus A):y).$$Consider the following short exact sequence.
\begin{align*}
0 & \longrightarrow \frac{S_A}{(I(G\setminus A)^{(2)}:u_i)}\longrightarrow \frac{S_A}{(I(G\setminus A):x)}\oplus\frac{S_A}{(I(G\setminus A):y)}\\ & \longrightarrow \frac{S_A}{(I(G\setminus A):x)+(I(G\setminus A):y)}\longrightarrow 0
\end{align*}
Applying depth Lemma \cite[Proposition 1.2.9]{bh} on the above exact sequence, it suffices to prove that
\begin{itemize}
\item [(a)] $\depth S_A/(I(G\setminus A):x) \geq \alpha_2(G)-1$,
\item [(b)] $\depth S_A/(I(G\setminus A):y) \geq \alpha_2(G)-1$, and
\item [(c)] $\depth S_A/((I(G\setminus A):x)+(I(G\setminus A):y))\geq \alpha_2(G)-2$.
\end{itemize}

To prove (a), note that
\begin{align*}
& (I(G\setminus A):x)=I(G\setminus(A\cup N_{G\setminus A}[x]))+({\rm the\ ideal\ generated\ by}\ N_{G\setminus A}(x))\\ & =I(G\setminus(A\cup N_G[x]))+({\rm the\ ideal\ generated\ by}\ N_{G\setminus A}(x)).
\end{align*}
Hence,
\[
\begin{array}{rl}
\depth S_A/(I(G\setminus A):x)=\depth S'/I(G\setminus(A\cup N_G[x])),
\end{array} \tag{9} \label{8}
\]
where $S'=\mathbb{K}[x_i: 1\leq i\leq n, x_i\notin A\cup N_G(x)]$. Obviously, $x$ is a regular element of $S'/I(G\setminus(A\cup N_G[x]))$. Therefor, Corollary \ref{spn} implies that
\[
\begin{array}{rl}
\depth S'/I(G\setminus(A\cup N_G[x]))\geq \alpha_2(G\setminus(A\cup N_G[x]))+1.
\end{array} \tag{10} \label{9}
\]
Assume that $A\subseteq N_G(x)\cup N_G(y)$. It then follows from Lemma \ref{packdel1} that$$\alpha_2(G\setminus(A\cup N_G[x]))\geq \alpha_2(G)-2.$$ Hence, we conclude from equality (\ref{8}) and inequality (\ref{9}) that$$\depth S_A/(I(G\setminus A):x)\geq \alpha_2(G)-1.$$Thus, to complete the proof of (a), we only need to show that $A\subseteq N_G(x)\cup N_G(y)$.

Let $z$ be an arbitrary variable in $A$ and suppose $z\notin N_G(x)\cup N_G(y)$. Then the only edge dividing $zu_i=zxy$ is $u_i$. In particular,
\[
\begin{array}{rl}
z\notin \big((u_1, \ldots, u_{i-1}):u_i).
\end{array} \tag{11} \label{10}
\]
Moreover, since $\{z,x\}$ is an independent of subset of vertices of $G$, we conclude that $C'=V(G)\setminus \{z,x\}$ is a vertex cover of $G$ with $zu_i=zxy\notin \mathfrak{p}_C^2$. Thus, $zu_i\notin I^{(2)}$. This means that $z\notin (I^{(2)}:u_i)$. This, together with (\ref{10}) implies that$$z\notin\big((I^{(2)}, u_1, \ldots, u_{i-1}):u_i\big),$$which is a contradiction. Therefore, $z\in N_G(x)\cup N_G(y)$. Hence, $A\subseteq N_G(x)\cup N_G(y)$ and this completes the proof of (a). The proof of (b) is similar to the proof of (a). We now prove (c).

Note that
\begin{align*}
& (I(G\setminus A):x)+(I(G\setminus A):y)=\\ & I(G\setminus(A\cup N_{G\setminus A}[x]\cup N_{G\setminus A}[y]))+({\rm the\ ideal\ generated\ by}\ N_{G\setminus A}[x]\cup N_{G\setminus A}[y])=\\ & I(G\setminus(A\cup N_G[x]\cup N_G[y]))+({\rm the\ ideal\ generated\ by}\ N_{G\setminus A}[x]\cup N_{G\setminus A}[y])=\\ & I(G\setminus(N_G[x]\cup N_G[y]))+({\rm the\ ideal\ generated\ by}\ N_{G\setminus A}[x]\cup N_{G\setminus A}[y]),
\end{align*}
Where the last equality follows from $A\subseteq \big(N_G(x)\cup N_G(y)\big)$. We conclude that
\[
\begin{array}{rl}
\depth S_A/((I(G\setminus A):x)+(I(G\setminus A):y))=\depth S''/I(G\setminus(N_G[x]\cup N_G[y])),
\end{array} \tag{12} \label{11}
\]
where $S''=\mathbb{K}\big[x_i: 1\leq i\leq n, x_i\notin N_G[x]\cup N_G[y]\big]$. Using Corollary \ref{spn} and Lemma \ref{packdel1}, we deuce that$$\depth S''/I(G\setminus(N_G[x]\cup N_G[y]))\geq \alpha_2(G\setminus(N_G[x]\cup N_G[y]))\geq \alpha_2(G)-2.$$Finally, the assertion of (c) follows from equality (\ref{11}) and the above inequality. This completes the proof of the theorem.
\end{proof}





\end{document}